\documentclass[a4paper]{amsart}
\usepackage[all]{xy}
\usepackage{amssymb}
\usepackage{amsthm}
\SelectTips{cm}{}

\theoremstyle{plain}
\newtheorem{thm}{Theorem}[section]
\newtheorem{cor}[thm]{Corollary}
\newtheorem{lem}[thm]{Lemma}
\newtheorem{prop}[thm]{Proposition}

\theoremstyle{definition}
\newtheorem{defn}[thm]{Definition}

\theoremstyle{remark}
\newtheorem{rem}[thm]{Remark}
\newtheorem{exmp}[thm]{Example}

\newcommand{\Z}{{\mathbb Z}}
\newcommand{\Q}{{\mathbb Q}}
\newcommand{\T}{{\mathcal{T}}}

\newcommand{\C}{\mathcal{C}}
\newcommand{\M}{\mathcal{M}}

\newcommand{\map}{\mathop{\textrm{\rm map}}}

\newcommand{\Hom}{\mathop{\textrm{\rm Hom}}}
\newcommand{\Ext}{\mathop{\textrm{\rm Ext}}}

\newcommand{\Cell}{\mathit{Cell}}

\numberwithin{equation}{section}

\begin{document}
\title{Cellularization of structures in stable homotopy categories}
\author{Javier J.~Guti\'errez}
\thanks{The author was supported by the MEC-FEDER grants MTM2007-63277 and MTM2010-15831, and by the
Generalitat de Catalunya as a member of the team 2009~SGR~119.}

\address{Departament d'\`Algebra i Geometria, Universitat de Barcelona, Gran Via de les Corts Catalanes 585, 08007 Barcelona, Spain}
\email{javier.gutierrez.math@gmail.com}

\begin{abstract}
We describe the formal properties of cellularization functors in triangulated categories and study the preservation of ring and module structures under these functors in stable homotopy categories in the sense of Hovey, Palmieri and Strickland, such as the homotopy category of spectra or the derived category of a commutative ring. We prove that cellularization functors preserve modules over connective rings but they do not preserve rings in general (even if the ring is connective or the cellularization functor is triangulated). As an application of these results, we describe the cellularizations of Eilenberg-Mac\,Lane spectra and compute all acyclizations in the sense of Bousfield of the integral Eilenberg-Mac\,Lane spectrum.

\end{abstract}
\maketitle

\section{Introduction}
Cellularization functors were first described in homotopy theory
by Farjoun in the category of pointed simplicial
sets~\cite{Far96}. Given two pointed simplicial sets $A$ and $X$,
one can construct in a functorial way another space $\Cell_A X$,
called the \emph{$A$\nobreakdash-cellularization} of $X$, that
contains the information on $X$ that can be ``built up'' from $A$.
A pointed map $f\colon X\rightarrow Y$ of fibrant simplicial sets
is an \emph{$A$-cellular equivalence} if the induced map
$$
f_*\colon\map\nolimits_*(A,X)\longrightarrow\map\nolimits_*(A,Y)
$$
is a weak equivalence. The simplicial sets of the form $\Cell_A X$
are called \emph{$A$-cellular}, and they form the smallest class
of spaces that contains $A$ and it is closed under weak
equivalences and pointed homotopy colimits. For every pointed
simplicial set $A$ there is an $A$-cellularization functor
$\Cell_A$ that is augmented and homotopy idempotent, and it gives
the best possible $A$-cellular approximation in the sense that the
natural map $\Cell_A X\rightarrow X$ is an $A$-cellular
equivalence. Its properties in the category of spaces have been
widely studied; see~\cite{Far96}, \cite{Cha96} and \cite{Cha97}.
Recently, cellularization functors have been also described in
algebraic categories such as that the category of
groups~\cite{FGS07}, the category of abelian groups~\cite{FG09},
\cite{CFGS09},  and the category of modules \cite{GRS09}.

This paper is devoted to the study of these functors in the
framework of triangulated categories. In the first part, we
describe the formal properties of cellularization functors and the
closure properties of cellular objects and cellular equivalences
in this setting.

Let $\T$ be a triangulated category with products and coproducts
and denote by $\Sigma$ the suspension or shift functor. Let $A$ be
any object in $\T$. A morphism $f\colon X\rightarrow Y$ is an
\emph{$A$-cellular equivalence} if the induced map
$$
f_*\colon\T(\Sigma^k A, X)\longrightarrow\T(\Sigma^k A, Y)
$$
is an isomorphism of abelian groups for all $k\ge 0$. An object $Z$ in $\T$ is \emph{$A$-cel\-lu\-lar} if the induced map
$$
g_*\colon \T(\Sigma^k Z, X)\longrightarrow\T(\Sigma^k Z, Y)
$$
is an isomorphism for every $A$-cellular equivalence $g\colon X\rightarrow Y$ and for all $k\ge 0$.

If the category $\T$ has an underlying right proper cellular or right proper combinatorial model, then
for any object $A$ there exists an $A$-cellularization functor $\Cell_A$ (see \cite[Ch.\ 5]{Hir03}, \cite{Bar10}). This functor is augmented and idempotent in $\T$, and assigns to every object $X$ in $\T$ an $A$-cellular object $\Cell_A X$ together with an $A$-cellular equivalence $\Cell_A X\rightarrow X$ that is terminal among morphisms from $A$-cellular objects to~$X$. The full subcategory of $A$-cellular objects is the image of the functor $\Cell_A$ and it is closed under cofibres, retracts and coproducts but not under desuspensions, in general.

When $\Cell_A$ is triangulated, i.e., it commutes with $\Sigma$,
then $A$-cellular objects are also closed under fibres and
extensions, and for every $X$ in $\T$ there is a triangle
$$
\Cell_A X\longrightarrow X\longrightarrow P_A X\longrightarrow \Sigma\Cell_A X,
$$
where $P_A$ denotes the $A$-nullification functor (see Definition~\ref{defnull}).
In the homotopy category of spectra, Bousfield's $E_*$\nobreakdash-acyclizations~\cite{Bou79a} and $[E,-]_*$\nobreakdash-coloca\-li\-za\-tions~\cite{Bou79b} are examples of functors with all these closure properties.

In the general case, we show that choosing a fiber $F_A X$ of the nullification map $X\rightarrow P_A X$ for every $X$, and using the universal  property of the localization functor $P_A$, we obtain an augmented idempotent functor $F_A$ and a triangle
$$
F_A X\longrightarrow X\longrightarrow P_A X\longrightarrow \Sigma F_A X
$$
for every $X$ in~$\T$. Although in general $F_A$ is not $\Cell_A$, it is often closely related to it.
In the homotopy category of spectra, Chorny has shown that assuming certain large-cardinal axiom (Vop\v{e}nka's principle), $F_A$ is always of the form $\Cell_E X$ for some spectrum $E$. In the category of
spaces the same is true without assuming Vop\v{e}nka's principle, and an explicit construction of $E$ can be found in \cite{CPS04}. Here we prove that if $\T(\Sigma^{-1}A, X)=0$ then $F_A X\cong \Cell_{\Sigma^{-1}A}X$. We also give a characterization of the cellularization and nullification functors that fit into a triangle of the form
$$
\Cell_{A} X\longrightarrow X\longrightarrow P_{A} X\longrightarrow \Sigma \Cell_A X.
$$
If this is the case, then the full subcategories of $A$-cellular objects and $\Sigma A$\nobreakdash-null objects define, in a canonical way, a $t$-structure in the sense of \cite{BBD82} on the triangulated category $\T$, whose heart is trivial if either $P_A$ or $\Cell_A$ are triangulated.

In the second part of the paper we study the preservation of ring and module structures by cellularization functors in \emph{stable homotopy categories}~\cite{HPS97}. These are triangulated categories with a compatible closed symmetric monoidal structure. Examples of stable homotopy categories are the homotopy category of spectra and the derived category of a commutative ring with unit.

If $\T$ is a connective monogenic stable homotopy category, that is,
the unit $S$ of the monoidal structure is a small generator and
$\T(\Sigma^k S, X)=0$ for $k< 0$, we prove that cellularization
functors preserve modules over rings if either the ring is
connective or the functor is triangulated. However, rings are not
preserved in general by cellularization functors, even if the ring
is connective or the functor is triangulated. In~\ref{acycex} we exhibit an example in the homotopy category of spectra of a triangulated cellularization functor that does not send the ring spectrum $H\Z$ to a ring spectrum.  In Section~\ref{cellringsmods} we give
sufficient conditions for cellularization functors to preserve
ring objects.

As a consequence of our results on the cellularization of modules, we prove that cellularization functors in the homotopy category of
spectra preserve stable GEMs and that the cellularization of an Eilenberg-Mac\,Lane
spectrum $\Sigma^k HG$ is either zero or it has at most two nonzero homotopy groups in dimensions $k$ and $k-1$.
Finally, we study in more detail acyclizations in the sense of Bousfield~\cite{Bou79a} of Eilenberg-Mac\,Lane spectra and compute some particular examples.

\section{Cellularization and nullification in triangulated categories}

Let $\C$ be any category. A functor $C\colon \C\rightarrow\C$ is called \emph{augmented} if it is
equipped with a natural transformation $\mu\colon C\rightarrow {\rm Id}$. A functor $L\colon\C\rightarrow\C$ is called
\emph{coaugmented} if it is equipped with a natural transformation $\eta\colon {\rm Id}\rightarrow L$.
An augmented (resp.\ coaugmented) functor is \emph{idempotent} if $\mu_{CX}=C\mu_X$ (resp.\ $\eta_{LX}=L\eta_X$) and  $\mu_{CX}$ (resp.\ $\eta_{LX}$) is an isomorphism for every $X$ in $\C$.

An augmented idempotent functor $C$ is also called a
\emph{colocalization functor}. The closure under isomorphisms of
the objects in the image of $C$ form the class of
\emph{$C$\nobreakdash-colocal objects} and the maps $f$ such that $C(f)$ is an
isomorphism are called \emph{$C$-colocal equivalences}. Similarly,
a coaugmented idempotent functor is called a \emph{localization
functor}. The closure under isomorphisms of the objects in the
image of~$L$ are the \emph{$L$-local objects} and the maps $f$
such that $L(f)$ is an isomorphism are called \emph{$L$-local
equivalences}.

Cellularization and nullification functors are particular examples of augmented and coaugmented idempotent functors, respectively. In this section, we define these functors and describe their formal properties in the context of triangulated categories.

A \emph{triangulated category} is an additive category $\T$ with an invertible endofunctor $\Sigma\colon \T\rightarrow \T$ together with
a collection of diagrams of the form
$$
X\longrightarrow Y\longrightarrow Z\longrightarrow \Sigma X
$$
called \emph{distinguished triangles} that satisfy certain axioms
(see for example \cite{Ver77}, \cite{Nee01}). We will refer to
distinguished triangles simply as \emph{triangles}. Given a
triangle
$$
\Sigma^{-1}Z\longrightarrow X\stackrel{f}{\longrightarrow} Y\longrightarrow Z,
$$
we say that $\Sigma^{-1}Z$ is the \emph{fibre} of $f$ and $Z$ is the \emph{cofibre} of $f$.
For any two objects $X$ and $Y$ in a triangulated category $\T$, we will denote by $\T(X,Y)$ the abelian group of morphisms from $X$ to $Y$. Note that since any triangulated category is additive it has a zero object denoted by 0.

Let $\mathcal{L}$ be a full subcategory of $\T$ and let $X\rightarrow Y\rightarrow Z\rightarrow \Sigma X$ be any triangle in~$\T$. We say that $\mathcal{L}$ is:
\begin{itemize}
\item[{\rm (i)}] \emph{Closed under fibres} if when $Y$ and $Z$ are in $\mathcal{L}$ then $X$ is also in $\mathcal{L}$.
\item[{\rm (ii)}] \emph{Closed under cofibres} if when $X$ and $Y$ are in $\mathcal{L}$ then $Z$ is also in $\mathcal{L}$.
\item[{\rm (iii)}] \emph{Closed under extensions} if when $X$ and $Z$ are in $\mathcal{L}$ then $Y$ is also in $\mathcal{L}$.
\end{itemize}
Similarly, if $\mathcal{S}$ is a class of morphisms in $\T$ and
$$
\xymatrix{
X\ar[r]\ar[d]_{f} & Y\ar[r]\ar[d]_{g} & Z\ar[d]_{h} \ar[r]& \Sigma X \ar[d]_{\Sigma f} \\
X' \ar[r] & Y' \ar[r] & Z'\ar[r] & \Sigma X'
}
$$
is any morphism of triangles in $\T$, we say that $\mathcal{S}$ is:
\begin{itemize}
\item[{\rm (i)}] \emph{Closed under fibres} if when $g$ and $h$ are in $\mathcal{S}$ then $f$ is also in $\mathcal{S}$.
\item[{\rm (ii)}] \emph{Closed under cofibres} if when $f$ and $g$ are in $\mathcal{S}$ then $h$ is also in $\mathcal{S}$.
\item[{\rm (iii)}] \emph{Closed under extensions} if when $f$ and $h$ are in $\mathcal{S}$ then $g$ is also in $\mathcal{S}$.
\end{itemize}

\begin{defn}
Let $\T$ be any triangulated category and let $A$ be any object in $\T$.
\begin{itemize}
\item[{\rm (i)}] A morphism $f\colon X\rightarrow Y$ is an \emph{$A$-cellular equivalence} if the induced map
$$
f_*\colon \T(\Sigma^k A, X)\longrightarrow \T(\Sigma^k A, Y)
$$
is an isomorphism of abelian groups for all $k\ge 0$.
\item[{\rm (ii)}] An object $Z$ is \emph{$A$-cellular} if the induced map
$$
g_*\colon \T(\Sigma^k Z, X)\longrightarrow \T(\Sigma^k Z, Y)
$$
is an isomorphism for every $A$-cellular equivalence $g\colon X\rightarrow Y$ and all $k\ge 0$.
\item[{\rm (iii)}] An \emph{$A$-cellularization} of $X$ is an $A$-cellular equivalence
$\widetilde{X}\rightarrow X$, where $\widetilde{X}$ is $A$-cellular.
\end{itemize}
\label{defcell} An \emph{$A$-cellularization functor} is a
colocalization functor $(\Cell_A, c)$ such that for every object
$X$ in $\T$ the natural morphism $c_X\colon \Cell_A X\rightarrow
X$ is an $A$-cellularization.
\end{defn}

\begin{defn}
Let $\T$ be any triangulated category and let $A$ be any object in $\T$.
\begin{itemize}
\item[{\rm (i)}] An object $X$ is \emph{$A$-null} if $\T(\Sigma^k A, X)=0$ for every $k\ge 0$.
\item[{\rm (ii)}] A morphism $f\colon X\rightarrow Y$ is an \emph{$A$-null equivalence} if the induced map
$$
f^*\colon \T(\Sigma^k Y,Z)\longrightarrow \T(\Sigma^k X, Z)
$$
is an isomorphism of abelian groups for every $A$-null object $Z$ and all $k\ge 0$.
\item[{\rm (iii)}] An \emph{$A$-nullification} of $X$ is an $A$-null equivalence
$X\rightarrow \widehat{X}$, where $\widehat{X}$ is $A$-null.
\end{itemize}
\label{defnull}
An \emph{$A$-nullification functor} is a localization functor $(P_A, l)$ such that for every object $X$ in $\T$, the
morphism $l_X\colon X\rightarrow P_A X$ is an $A$-nullification.
\end{defn}

We will assume that for every object $A$ in $\T$ there exist the
$A$\nobreakdash-cellu\-la\-ri\-za\-tion functor $\Cell_A$ and the
$A$-nullification functor $P_A$. This is guaranteed if the
triangulated category admits a suitable model where we can
construct cellularization and nullification functors. There are
two main classes of model categories where localizations and
colocalizations are always known to exist. These are the proper
cellular model categories~\cite[Theorem 5.1.1 and Theorem
4.1.1]{Hir03} and the proper combinatorial model
categories~\cite{Bar10}.
Hence, for us $\T={\rm Ho}(\M)$ will be the homotopy category of a stable model category $\M$ belonging to either of these two classes.

Our main examples are when $\M$ is the model category of symmetric spectra \cite{HSS00} or $\M$ is the model category of (unbounded) chain complexes
of $R$-modules, where $R$ is a commutative ring with unit \cite{Hov99}. In the first case $\T={\rm Ho}^s$ is the ordinary stable homotopy category, and in the second case $\T=\mathcal{D}(R)$ is the derived category of~$R$.

The cellularization map is characterized by two universal properties:
\begin{itemize}
\item[{\rm (i)}] If $f\colon Y\rightarrow X$ is an $A$-cellular equivalence, then there exists a unique $g\colon \Cell_A X\rightarrow Y$
such that the following diagram commutes in $\T$:
$$
\xymatrix{
\Cell_A X \ar[dr]_{g} \ar[r]^-{c_X} & X \\
 & Y.\ar[u]_f
}
$$
\item[{\rm (ii)}] If $f\colon Z\rightarrow X$ is a morphism where $Z$ is $A$-cellular, then there exists a unique $h\colon Z\rightarrow \Cell_A X$
such that the following diagram commutes in $\T$:
$$
\xymatrix{
\Cell_A X  \ar[r]^-{c_X} & X \\
 & Z.\ar[ul]^{h}\ar[u]_f
}
$$
\end{itemize}
The nullification map is also characterized by two universal properties that are stated dually.

From these universal properties, it follows that the $A$-cellular objects are those $X$ such that $\Cell_A X\cong X$,  and the $A$-cellular equivalences are precisely those morphisms $f$ such that $\Cell_A f$ is an isomorphism. Moreover, if the cellularization of an object exists, then it is unique up to isomorphism,
i.e., if we have an $A$\nobreakdash-cellular equivalence $Y\rightarrow X$ such that $Y$ is $A$\nobreakdash-cellular, then $Y\cong\Cell_A X$. The dual statements hold in the case of nullifications.

Cellular objects and cellular equivalences characterize cellularization functors in the following sense:
if two cellularization functors $\Cell_A$ and $\Cell_B$ have the same class of cellular objects or the same class of cellular equivalences,
then $\Cell_A X\cong \Cell_B X$ for every object $X$ in $\T$. Similarly, null objects and null equivalences characterize nullification
functors.

The following results gather some closure properties of cellular objects, null objects and equivalences. 

\begin{lem}
Let $A$ be any object in $\T$. Then the following hold:
\begin{itemize}
\item[{\rm (i)}] The full subcategory of $A$-cellular objects is closed under retracts, cofibres and coproducts, and the full subcategory of $A$-null objects is closed under retracts, fibres, extensions and products.
\item[{\rm (ii)}] The class of $A$-cellular equivalences is closed under fibres and the class of $A$-null equivalences is closed under cofibres and extensions.
\end{itemize}
\label{suseq}
\end{lem}
\begin{proof}
In any triangulated category $\T$, every triangle $X\rightarrow Y\rightarrow Z\rightarrow \Sigma X$ gives rise to two long exact sequences of abelian groups
\begin{gather}\notag
\cdots\longrightarrow \T(\Sigma E, Z)\longrightarrow \T(E, X)\longrightarrow \T(E, Y) \longrightarrow \T(E, Z)\longrightarrow
\T(\Sigma^{-1} E, X)\longrightarrow\cdots \\ \notag
\cdots\longleftarrow \T(\Sigma^{-1} Z, E)\longleftarrow \T(X, E)\longleftarrow \T(Y, E) \longleftarrow \T(Z, E)\longleftarrow
\T(\Sigma X, E)\longleftarrow\cdots
\end{gather}
for any object $E$ in $\T$. It is easy to check that the results follow from the definitions and these long exact sequences together with the five lemma.
\end{proof}

\begin{cor}
Let $A$ be any object in $\T$. Then the following hold:
\begin{itemize}
\item[{\rm (i)}] If $X$ is $A$-cellular, then $\Sigma^k X$ is $A$-cellular and $X$ is $\Sigma^{-k} A$-cellular for all $k\ge 0$,
and $\Sigma^ k X$ is $\Sigma^k A$-cellular for all $k\in\Z$.
\item[{\rm (ii)}] If $f$ is an $A$-cellular equivalence, then $\Sigma^{-k} f$ is an $A$-cellular equivalence and
$f$ is a $\Sigma^k A$-cellular equivalence for all $k\ge 0$, and $\Sigma^k f$ is a $\Sigma^k A$-cellular equivalence for all $k\in\Z$.
\item[{\rm (iii)}] If $X$ is $A$-null, then $\Sigma^k X$ is $A$-null and $X$ is $\Sigma^{-k} A$-null for all $k\le 0$,
and $\Sigma^ k X$ is $\Sigma^k A$-null for all $k\in\Z$.
\item[{\rm (iv)}] If $f$ is an $A$-null equivalence, then $\Sigma^{-k} f$ is an $A$-null equivalence and
$f$ is a $\Sigma^{k} A$-equivalence for all $k\le 0$, and $\Sigma^k f$ is a $\Sigma^k A$-equivalence for all $k\in\Z$.
\end{itemize}
\label{corolcel}
\end{cor}
\begin{proof}
Use the triangle $X\rightarrow X\rightarrow 0\rightarrow \Sigma X$ and Lemma~\ref{suseq}.
\end{proof}

\begin{cor}
If $A$ and $X$ are any two objects in $\T$, then
$$
\Cell_A\Sigma^k X\cong \Sigma^k\Cell_{\Sigma^{-k}A}X \quad\mbox{and}\quad P_A\Sigma^k X\cong \Sigma^kP_{\Sigma^{-k}A}X
$$
for all $k\in\Z$.
\label{comcor}
\end{cor}
\begin{proof}
The map $c\colon\Cell_{\Sigma^{-k}A}X\rightarrow X$ is a $\Sigma^{-k}A$-cellular equivalence, so Corollary~\ref{corolcel} implies that
$\Sigma^k c$ is an $A$\nobreakdash-cellular equivalence and that $\Sigma^k \Cell_{\Sigma^{-k}A}X$ is $A$\nobreakdash-cellular, since
$\Cell_{\Sigma^{-k}A}X$ is $\Sigma^{-k}A$-cellular. Hence, $\Cell_A\Sigma^k X\cong \Sigma^k\Cell_{\Sigma^{-k}A}X$. The nullification part is proved by the same argument.
\end{proof}

\begin{prop}
Let $X\rightarrow Y\rightarrow Z\rightarrow \Sigma X$ be a triangle in $\T$ and $A$ any object in $\T$. If $\Cell_A Z=0$,
then the morphism $X\rightarrow Y$ is an $A$-cellular equivalence. If $P_A X=0$, then the morphism $Y\rightarrow Z$ is
an $A$-null equivalence.
\label{propos02}
\end{prop}
\begin{proof}For the first part, use the long exact sequence
$$
\cdots\longrightarrow\T(\Sigma A, X)\longrightarrow \T(A,X)\longrightarrow \T(A,Y)\longrightarrow \T(A,Z)\longrightarrow
\T(\Sigma^{-1} A, X) \longrightarrow\cdots
$$
and the fact that the map $0\rightarrow Z$ is an $A$-cellular equivalence. The second part is proved similarly. If
$E$ is an $A$-null object, then from the long exact sequence
$$
\cdots\longleftarrow \T(\Sigma^{-1} Z, E)\longleftarrow \T(X, E)\longleftarrow \T(Y, E) \longleftarrow \T(Z, E)\longleftarrow
\T(\Sigma X, E)\longleftarrow\cdots
$$
it follows that the map $Y\rightarrow Z$ is an $A$-null equivalence since the map $X\rightarrow 0$ is also an $A$-null
equivalence.
\end{proof}

An additive functor $F\colon \T\rightarrow \T'$ between two triangulated categories is called a \emph{triangulated functor} if there is a natural isomorphism $\phi_X\colon F(\Sigma X)\rightarrow \Sigma (FX)$
for every object $X$ in $\T$ such that for every triangle
$$
X\stackrel{u}{\longrightarrow} Y \stackrel{v}{\longrightarrow} Z\stackrel{w}{\longrightarrow} \Sigma X
$$
in $\T$ the diagram
$$
FX\stackrel{Fu}{\longrightarrow} FY \stackrel{Fv}{\longrightarrow} FZ\stackrel{\phi_X\circ Fw}{\longrightarrow} \Sigma (FX)
$$
is a triangle in $\T'$.

Cellularization and nullification functors are not triangulated in general, since they do not preserve triangles (see Example~\ref{expos}). However, by the universal property of the cellularization map, for every $A$ and $X$ in $\T$, there is a natural map
$$
\Sigma \Cell_A X\longrightarrow \Cell_A\Sigma X,
$$
since the suspension of an $A$-cellular object is again
$A$-cellular. The following theorem provides a characterization of
when this map is an isomorphism. If this is the case, we say that the
functor $\Cell_A$ \emph{commutes with suspension} (or,
equivalently, is triangulated).
\begin{thm}
Let $A$ be any object in $\T$. Then, the following are equivalent:
\begin{itemize}
\item[{\rm (i)}] $\Sigma \Cell_A X\cong\Cell_A\Sigma X$ for every $X$ in $\T$.
\item[{\rm (ii)}] $A$-cellular objects are closed under desuspensions.
\item[{\rm (iii)}] $A$-cellular equivalences are closed under suspensions.
\item[{\rm (iv)}] $\Cell_A X\cong \Cell_{\Sigma^k A}X$ for every $X$ in $\T$ and all $k\in\Z$.
\item[{\rm (v)}] A map $g\colon X\rightarrow Y$ in $\T$ is an $A$-cellular equivalence if the induced map
$$
g_*\colon \T(\Sigma^k A, X)\longrightarrow \T(\Sigma^k A, Y)
$$
is an isomorphism for all $k\in\Z$.
\item[{\rm (vi)}] If $X\rightarrow Y\rightarrow Z\rightarrow \Sigma X$ is a triangle in $\T$, then so is $\Cell_A X\rightarrow \Cell_A Y\rightarrow \Cell_A Z\rightarrow \Cell_A\Sigma X$.
\end{itemize}
\label{largo}
\end{thm}
\begin{proof}
The equivalence between (i), (ii) and (iii) is a direct consequence of the definitions. From (iii) it follows that the class of $A$-cellular equivalences and the class of $\Sigma^k A$-cellular equivalences coincide, implying (iv). Combining (iv) and Corollary~\ref{comcor} we deduce (i). Part (v) and (iii) are equivalent. To get (vi) from~(i), consider the following diagram
associated to a triangle $X\rightarrow Y\rightarrow Z$:
$$
\xymatrix{
\Sigma^{-1} \Cell_A Z\ar[r]\ar[d] & \Sigma^{-1}\Cell_A \Sigma X\ar[r]\ar[d] & F\ar[r]\ar[d] & \Cell_A Z \ar[r]\ar[d] & \Cell_A \Sigma X\ar[d]\\
\Sigma^{-1} Z\ar[r] & X\ar[r] & Y\ar[r] & Z \ar[r] & \Sigma X,
}
$$
where $F$ is a fibre of the map $\Cell_A Z\rightarrow \Cell_A \Sigma X$. Now, by part (i) we have that $\Sigma^{-1} \Cell_A Z\cong \Cell_A\Sigma^{-1} Z$ and
$\Sigma^{-1}\Cell_A \Sigma X\cong \Cell_A X$. Using Lemma \ref{suseq} we infer that $F$ is $A$-cellular and the map $F\rightarrow Y$
is an $A$-cellular equivalence, so $F\cong \Cell_A Y$. Finally, we can use the  triangle
$X\rightarrow X\rightarrow 0\rightarrow \Sigma X$ to deduce (i) from~(vi).
\end{proof}

\begin{rem}
Theorem \ref{largo} can be also stated in a dual form in terms of nullification functors (cf. \cite[Theorem 2.7]{CG05}).
\end{rem}

\section{Relation between cellularization and nullification}
For any given object $A$ in $\T$, there is a close relationship between the $A$-nulli\-fi\-ca\-tion functor and the
$A$-cellularization functor. While the first one ``kills'' all the $A$-information from a given object $X$, the second
one extracts the information from $X$ that can be ``built up'' from $A$. The following lemma illustrates this fact.
\begin{lem}
For any objects $A$ and $X$ in $\T$, we have that $\Cell_A P_A X\cong 0$ and $P_A \Cell_A X\cong 0$.
\label{lemtriv}
\end{lem}
\begin{proof}
Since $P_A X$ is $A$-null, $\T(\Sigma^ k A, P_A X)=0$ for all $k\ge 0$, so the map $P_A X\rightarrow 0$ is an $A$-cellular equivalence. This proves that
$\Cell_A P_A X=0$. Now, for the second part, it is enough to prove that the map $\Cell_A X\rightarrow 0$ is an $A$-null equivalence, i.e., that
$\T(\Sigma^k \Cell_A X, Z)=0$ for every $A$-null object $Z$ and all $k\ge 0$. But $Z$ being $A$-null is equivalent to the map $Z\rightarrow 0$ being an $A$-cellular
equivalence. Hence $\T(\Sigma^k \Cell_A X, Z)=0$ for all $k\ge 0$ since $\Sigma^k\Cell_A X$ is $A$-cellular for all $k\ge 0$.
\end{proof}

Cellularization and nullification functors do not preserve triangles in general.
In fact, as a consequence of Theorem \ref{thmnew}, for
every object $A$ in~$\T$, the functor $\Cell_A$ commutes with suspension if and only if $P_A$ commutes with suspension, and in this case $P_A$ and $\Cell_A$ fit into a triangle relating both functors. However, if
$X\rightarrow Y\rightarrow Z\rightarrow \Sigma X$ is a triangle and neither $P_A$ nor $\Cell_A$ are triangulated, it is possible to get another triangle keeping $X$ and taking the cellularization of $Z$ or keeping $Z$ and taking the nullification of $X$. The following result is a version of the \emph{fibrewise localization} of \cite[1.F.1]{Far96} for triangulated categories.

\begin{prop}
Let $A$ be any object in $\T$ and let $X\rightarrow Y\rightarrow Z\rightarrow \Sigma X$ be any triangle.
\begin{itemize}
\item[{\rm (i)}] \emph{(Cofibrewise cellularization)} There is a
morphism of triangles
$$
\xymatrix{
X\ar[r]\ar@{=}[d] & Y' \ar[r]\ar[d] & \Cell_A Z \ar[d] \ar[r] & \Sigma X\ar@{=}[d]\\
X\ar[r] & Y\ar[r] & Z\ar[r] & \Sigma X,
}
$$
where the map $Y'\rightarrow Y$ is an $A$-cellular equivalence.
\item[{\rm (ii)}] \emph{(Fibrewise nullification)} There is a morphism of triangles
$$
\xymatrix{
\Sigma^{-1} Z\ar@{=}[d] \ar[r] & X\ar[r]\ar[d] & Y \ar[r]\ar[d] & Z \ar@{=}[d] \\
\Sigma^{-1} Z \ar[r] & P_A X\ar[r] & Y''\ar[r] & Z,
}
$$
where the map $Y\rightarrow Y''$ is an $A$-null equivalence.
\end{itemize}
\label{fibcof}
\end{prop}
\begin{proof}
We prove part (i). The second part is proved in the same way. Consider the following diagram
$$
\xymatrix{
X\ar@{=}[d] &  & \Cell_A Z \ar[r]\ar[d] & \Sigma X\ar@{=}[d]\\
X\ar[r] & Y\ar[r] & Z\ar[r] & \Sigma X
}
$$
and  complete it to a morphism of triangles by defining $Y'$ as a fiber of the composition $\Cell_A Z\rightarrow Z\rightarrow \Sigma X$. The fact that the map $Y'\rightarrow Y$ is an $A$-cellular equivalence follows from Lemma~\ref{suseq}.
\end{proof}

It is also possible to relate cofibrewise $A$-cellularization with
$\Sigma^{-1}A$-cellularization and fibrewise $A$-nullification
with $\Sigma A$-nullification.
\begin{prop}
Let $A$ be any object in $\T$ and let $X\rightarrow Y\rightarrow Z\rightarrow \Sigma X$ be any triangle. Then the following hold:
\begin{itemize}
\item[{\rm (i)}] There is a morphism of triangles
$$
\xymatrix{
\Cell_{\Sigma^{-1}A}X \ar[r]\ar[d] & Y' \ar[r]\ar[d] & \Cell_A Z\ar[d]\ar[r] & \Sigma\Cell_{\Sigma^{-1}A} X \ar[d] \\
X\ar[r] & Y\ar[r] & Z \ar[r] & \Sigma X,
}
$$
where $Y'$ is $\Sigma^{-1}A$-cellular and the map $Y'\rightarrow Y$ is an $A$-cellular equivalence.
\item[{\rm (ii)}] There is a morphism of triangles
$$
\xymatrix{
\Sigma^{-1} Z \ar[r] \ar[d] & X\ar[r]\ar[d] & Y \ar[r]\ar[d] & Z \ar[d] \\
\Sigma^{-1}P_{\Sigma A} Z\ar[r] & P_A X\ar[r] & Y''\ar[r] & P_{\Sigma A} Z,
}
$$
where $Y''$ is $\Sigma A$-null and the map $Y\rightarrow Y''$ is an $A$-null equivalence.
\end{itemize}
\end{prop}

\begin{proof}
For part (i), consider the following diagram
$$
\xymatrix{
\Cell_{\Sigma^{-1}A}X \ar[d] & & \Cell_A Z\ar[r]\ar[d]_{c_Z} & \Cell_A \Sigma X\ar[d]^{c_{\Sigma X}} \\
X\ar[r] & Y\ar[r] & Z \ar[r] & \Sigma X
}
$$
and complete it to a morphism of triangles by defining $Y'$ as a fiber of the map $\Cell_A Z\rightarrow \Cell_A \Sigma X$ and by using that $\Cell_{\Sigma^{-1}A}X\cong\Sigma^{-1}\Cell_{A}\Sigma X$ by Corollary~ \ref{comcor}. Thus, we have a triangle
$$
\Cell_{\Sigma^{-1}A}X\longrightarrow Y'\longrightarrow \Cell_A Z\longrightarrow
\Sigma\Cell_{\Sigma^{-1}A} X.
$$
Now, $\Sigma^{-1}\Cell_A Z$ is $\Sigma^{-1}A$-cellular, by
Corollary \ref{corolcel}, and so $Y'$ is $\Sigma^{-1}A$-cellular
by Lemma \ref{suseq}. Since $c_Z$ and $c_{\Sigma X}$ are both
$A$-cellular equivalences, so is the map $Y'\rightarrow Y$, again
by Lemma \ref{suseq}. The object $Y''$ in part (ii) is defined as the cofibre
of the map $P_A\Sigma^{-1}Z\rightarrow P_A X$ and the induced map
$Y\rightarrow Y''$ is an $A$-null equivalence by a similar
argument as above.
\end{proof}

\begin{defn}
We say that a localization functor in $\T$ is \emph{semiexact} if the full subcategory of local objects
is closed under extensions, fibres and products. Dually, a colocalization functor is \emph{semiexact} if the full subcategory of colocal objects is closed under extensions, cofibres and coproducts.
\end{defn}

As we have seen in Lemma~\ref{suseq} all nullification functors are semiexact. However, note that $\Cell_A$ is not semiexact in general (see Example \ref{notqe}). Semiexact functors \emph{almost} preserve triangles.

\begin{cor}
Let $X\rightarrow Y\rightarrow Z\rightarrow \Sigma X$ be any triangle and let $A$ be any object in~$\T$. If~$Z$ is $A$-null, then the nullification functor $P_A$ preserves the triangle.
If the cellularization functor $\Cell_A$ is semiexact and $X$ is $A$-cellular, then $\Cell_A$ preserves the triangle.
\label{cor01}
\end{cor}
\begin{proof}
Using the fact that $P_A$ is always semiexact, it follows that the object $Y''$ in Proposition~\ref{fibcof}(ii) is also $A$-null. Similarly, if $Cell_A$ is semiexact then the object $Y'$ in Proposition~\ref{fibcof}(i) is $A$-cellular.
\end{proof}

The following theorem gives a characterization of the nullification and cellularization functors related by a triangle.
\begin{thm}
Let $A$ and $X$ be two objects in $\T$. Then the following hold:
\begin{itemize}
\item[{\rm (i)}] There is a triangle $\Cell_A X\rightarrow X\rightarrow P_A X\rightarrow \Sigma\Cell_A X$
if and only if the morphism of abelian groups $\T(\Sigma^{-1} A, \Cell_A X)\rightarrow \T(\Sigma^{-1}A ,X)$ is injective (e.g. if $\T(\Sigma^{-1} A, \Cell_A X)=0$).
\item[{\rm (ii)}] There is a triangle $\Cell_A X\rightarrow X\rightarrow P_{\Sigma A} X\longrightarrow \Sigma\Cell_A X$
if and only if the morphism $\T(A,X)\rightarrow \T(A, P_{\Sigma A}X)$ is the zero map (e.g. if $\T(A,X)=0$).
\end{itemize}
\label{thmnew}
\end{thm}
\begin{proof}
Let $C$ be a cofibre of the cellularization map $c_X\colon \Cell_A
X\rightarrow X$. By Lemma~\ref{lemtriv}, we have that $P_A \Cell_A
X\cong 0$, therefore, Proposition \ref{propos02} implies that the
map $X\rightarrow C$ is an $A$-null equivalence. Associated to
this triangle we have a long exact sequence of abelian groups
\begin{multline}\notag
\cdots\rightarrow\T(\Sigma A, \Cell_A X)\rightarrow \T(\Sigma A, X)\rightarrow \T(\Sigma A, C)
\rightarrow \T(A, \Cell_A X) \rightarrow \T(A, X)\\
\rightarrow \T(A, C)
\rightarrow \T(\Sigma^{-1}A, \Cell_A X)\rightarrow \T(\Sigma^{-1} A, X)\rightarrow \T(\Sigma^{-1} A, C)\rightarrow \cdots
\end{multline}
Since the cellularization map is always an $A$-cellular equivalence
$$
\T(\Sigma^k A, \Cell_A X)\cong\T(\Sigma^k A, X)
$$
for every $k\ge 0$, so $\T(\Sigma^k A, C)=0$ for
every $k\ge 1$ and thus $C$ is $\Sigma A$-null. It follows from the exact sequence that $C$ is $A$-null if and only of
the map
$$
\T(\Sigma^{-1} A, \Cell_A X)\longrightarrow \T(\Sigma^{-1}A ,X)
$$
is injective. This proves the first part.

To prove the second part, choose a fibre $F$ of the nullification map $l_X\colon X\rightarrow P_{\Sigma A}X$. Then, by a similar argument
as before, $F$ is $A$-cellular and the map $F\rightarrow X$ is a $\Sigma A$-cellular equivalence.
Now, from the exact sequence
$$
\cdots\longrightarrow \T(\Sigma A, P_{\Sigma A} X)\longrightarrow \T(A, F)\longrightarrow \T(A, X)\longrightarrow \T(A, P_{\Sigma A} X)\longrightarrow\cdots
$$
and using the fact that $P_{\Sigma A}X$ is $\Sigma A$-null, it follows that $\T(A,F)\cong \T(A,X)$, and hence
$F\rightarrow X$ is an $A$-cellular equivalence if and only if $\T(A,X)\rightarrow \T(A, P_{\Sigma A}X)$ is the zero map.
\end{proof}
The hypotheses of this theorem are satisfied for example when $\Cell_A$ or $P_A$ commute with suspension. Note that if
$\Cell_A$ commute with suspension, so does $P_A$ and viceversa. Indeed, if this is the case, then for all $k\in\Z$ there
is a triangle
$$
\Cell_{\Sigma^k A} X\longrightarrow X\longrightarrow P_{\Sigma^{k}A}X\longrightarrow\Sigma \Cell_{\Sigma^k A}X.
$$

\begin{rem}
Note that not every nullification and cellularization functor related by a triangle as above commutes with suspension. For example, if $\T$ is the homotopy category of spectra and $S$ is the sphere spectrum, then we have a triangle (see Example \ref{expos}):
$$
\Cell_{\Sigma^k S} X\longrightarrow X\longrightarrow P_{\Sigma^k S} X\longrightarrow \Sigma\Cell_{\Sigma^k S} X
$$
for every $X$ and all $k\in\Z$, but neither $\Cell_{\Sigma^k S}$ nor $P_{\Sigma^k S}$ commute with suspension.
\end{rem}

\section{$t$-structures associated to nullifications}
\label{tstruct}
In this section we show that there is a canonical way to associate to any $A$\nobreakdash-nul\-li\-fi\-ca\-tion functor in a triangulated category a
$t$-structure in the sense of \cite{BBD82}. This $t$-structure is defined by the $\Sigma A$-null objects and the colocal objects of a
colocalization functor associated with $P_A$ defined in the following way.
Let $(P_A, l)$ be a nullification functor in $\T$. For every $X$, choose a fibre $F_A X$ of the nullification map $l_X\colon X\rightarrow P_A X$. Thus, we have a triangle
$$
F_A X\stackrel{c_X}{\longrightarrow} X\stackrel{l_X}{\longrightarrow} P_A X\longrightarrow \Sigma F_A X.
$$
If we apply $P_A$ to this triangle, since $P_A X$ is $A$-null, Corollary~\ref{cor01} implies that
$$
P_A F_A X\longrightarrow P_A X\longrightarrow P_A P_A X\longrightarrow \Sigma P_AF_A X
$$
is a triangle, and therefore $P_A F_A X=0$. For each morphism $f\colon X\rightarrow Y$, we choose now a morphism $F_A f\colon F_A X\rightarrow F_A Y$ such that the following diagram commutes
$$
\xymatrix{
\Sigma^{-1} P_A X \ar[r]\ar[d]_{\Sigma^{-1}P_A f} & F_A X\ar[d]_{F_A f}\ar[r] & X\ar[r]^{l_X}\ar[d]_{f} & P_A X \ar[d]_{P_A f}\\
\Sigma^{-1} P_A Y \ar[r] & F_A Y\ar[r] & Y\ar[r]_{l_Y} & P_A Y.
}
$$
The morphism $F_A f$ is unique, since $P_A F_A X=0$ implies that
$\T(F_A X, \Sigma^k P_A X)=0$ for all $k\le 0$, and therefore $F_A
Y\rightarrow Y$ induces an isomorphism $\T(F_A X, F_A Y)\cong
\T(F_A X, Y)$. This proves that $F_A$ is a functor. Moreover,
since $F_AF_A X$ is the fibre of $F_A X\rightarrow P_A F_A X$ by
definition and $P_AF_A X=0$, the natural morphism $c_{F_A X}\colon F_A F_A X\rightarrow F_A X$ is
an isomorphism. We have that $c_{F_A X}=F_Ac_X$, since $c_X$ induces a bijection $\T(F_AF_A X, F_A X)\cong \T(F_AF_A X, X)$. Thus $F_A$ is idempotent and hence a
colocalization functor in $\T$.

Similarly, starting with a cellularization functor $\Cell_A$, we can construct a functor $C_A$ by taking $C_A X$ as a
cofibre of the cellularization map $\Cell_A X\rightarrow X$.
By the same argument as above, and using
Corollary~\ref{cor01}, it follows that when $\Cell_A$ is
semiexact, the functor $C_A$ is coaugmented and idempotent, i.e.,
a localization functor.

\begin{defn}
\label{assoc_loc_coloc}
Let $\T$ be a triangulated category and $A$ any object in $\T$. The functor $F_A$ defined above is called the \emph{colocalization functor associated to $P_A$}. If $\Cell_A$ is semiexact, then the functor $C_A$ is called the
\emph{localization functor associated to $\Cell_A$}.
\end{defn}

It follows that an object $X$ is $F_A$-colocal if and only if $P_A X=0$ and is $A$-null
if and only if $F_A X=0$. Dually, if $\Cell_A$ is semiexact, then an object
$X$ is $C_A$-local if and only if $\Cell_A X=0$, and is
$A$-cellular if and only if $C_A X= 0$.

\begin{lem}
If $X$ is $F_A$-colocal, then $\Sigma^k X$ is $F_A$-colocal for every $k\ge 0$. If $\Cell_A$ is semiexact and $X$ is $C_A$-local,
then $\Sigma^k X$ is $C_A$-local for every $k\le 0$.
\label{desus}
\end{lem}
\begin{proof}
If $X$ is $F_A$-colocal, then $P_A X=0$, so $P_A\Sigma^k X=0$ for all $k\ge 0$. Hence
$\Sigma^k X\cong F_A\Sigma^k X$. The second statement is proved similarly.
\end{proof}
\begin{prop}
For every object $A$ in $\T$ the functor $F_A$ is semiexact. If $\Cell_A$ is semiexact, then so is $C_A$.
\label{proposqx}
\end{prop}
\begin{proof}
We prove only the first claim. The second one is proved by the same argument.
We need to prove that $F_A$-colocal objects are closed under extensions, cofibres and coproducts.

Let $X\rightarrow Y\rightarrow Z\rightarrow\Sigma X$ be a triangle such that $F_A X\cong X$ and $F_A Z\cong Z$.
Consider the following diagram in which the columns and the central row are triangles
$$
\xymatrix{
F_A X\ar[r]\ar[d] & F_A Y\ar[r]\ar[d] & F_A Z \ar[d]\\
X\ar[r]\ar[d] & Y \ar[d]\ar[r] & Z \ar[d] \\
P_A X\ar[r] & P_A Y\ar[r] & P_A Z.
}
$$
Since $F_A X\cong X$ and $F_A Z\cong Z$, we have that $P_A X=0$ and $P_A Z=0$. Now, by Proposition~\ref{propos02},
the map $Y\rightarrow Z$ is an $A$-null equivalence. Hence $P_A Y\cong P_A Z=0$ and hence $F_A Y\cong Y$. Second, assume that $X$ and $Y$ are $F_A$-colocal. Then $P_A X=0$ and $P_A Y=0$ and the same argument implies that $P_A Z=0$, so $F_A Z\cong Z$.

Now, let $\{X_i\}_{i\in\mathcal{I}}$ be a family of $F_A$-colocal objects. Then $P_A X_i=0$ for all $i\in\mathcal{I}$ and hence $P_A \left(\coprod_{i\in\mathcal{I}} X_i\right)=0$, since the class of $A$-null equivalences is closed under coproducts. Therefore $\coprod_{i\in\mathcal{I}} X_i$ is $F_A$-colocal.
\end{proof}

\begin{cor}
Let $A$ be any object in $\T$. Then, $\Cell_A$ is semiexact if and only if for every object $X$ there is a triangle $\Cell_A X\rightarrow X\rightarrow P_A X\rightarrow \Sigma \Cell_A X$.
\end{cor}
\begin{proof}
If $\Cell_A$ is semiexact, then the associated functor $C_A$ is a localization functor and the map $X\rightarrow C_A X$ is an $A$-null equivalence. An object $X$ is $C_A$-local if and only if $C_A X\cong X$ or
equivalently $\Cell_A X=0$. But $\Cell_A X=0$ if and only if $P_A X\cong X$, since the map $X\rightarrow 0$ being an $A$-cellular equivalence
is the same as the object $X$ being $A$-null. Thus, the localization functors $C_A$ and $P_A$ have the same local objects and therefore they are equivalent. For the converse use Proposition~\ref{proposqx}.
\end{proof}

We recall from \cite[Definition 1.3.1]{BBD82} the notion of a
$t$-structure on a triangulated category $\T$.
\begin{defn}
A \emph{$t$-structure} on $\T$ is a pair of full subcategories $(\T^{\le 0}, \T^{\ge 0})$ such that,
denoting $\T^{\le n}=\Sigma^{-n}\T^{\le 0}$ and $\T^{\ge n}=\Sigma^{-n}\T^{\ge 0}$, the following hold:
\begin{itemize}
\item[{\rm (i)}] If $X$ is any object in $\T^{\le 0}$ and $Y$ is any object in $\T^{\ge 1}$, then $\T(X,Y)=0$.
\item[{\rm (ii)}] $\T^{\le 0}\subseteq \T^{\le 1}$ and $\T^{\ge 1}\subseteq \T^{\ge 0}$.
\item[{\rm (iii)}] For every object $X$ in $\T$, there is a triangle
$$
U\longrightarrow X\longrightarrow V\longrightarrow \Sigma U,
$$
where $U$ is an object in $\T^{\le 0}$ and $V$ is an object in $\T^{\ge 1}$.
\end{itemize}
The \emph{heart} of the $t$-structure is the full subcategory $\T^{\ge 0}\cap\T^{\le 0}$. The heart is always an abelian subcategory of $\T$.
\end{defn}

\begin{prop}
For any object $A$ in $\T$ the full subcategory of $\Sigma A$-null objects and the full subcategory of $F_A$-colocal objects
define a $t$-structure on $\T$.
\label{t-struc}
\end{prop}
\begin{proof}
Let $\T^{\ge 0}$ be the full subcategory of $\Sigma A$-null
objects and $\T^{\le 0}$ the full subcategory of $F_A$-colocal
objects. Observe that $\T^{\ge 1}$ is the full subcategory of
$A$-null objects. The functor $F_A$ is defined by choosing a fibre
of the nullification map, hence condition (iii) of the definition
of a $t$-structure holds since for every $X$ there is a triangle
$$
F_A X\longrightarrow X\longrightarrow P_A X\longrightarrow \Sigma F_A X.
$$
Since the cofibre of this triangle is $A$-null, applying $P_A$ we get another triangle by Corollary~\ref{cor01}. Therefore
$P_A F_A X=0$ for every $X$ in $\T$. This means that the map $F_A X\longrightarrow 0$ is an $A$-null equivalence for every $X$ and
thus $\T(F_A X, P_A Y)=0$ proving condition (i). Now, condition (ii) follows easily since $A$-null objects are closed under
desuspensions and $F_A$-colocal objects are closed under suspensions by Lemma~\ref{desus}.
\end{proof}
For example, if for every $X$ in $\T$ we have that $\T(\Sigma^{-1}A,\Cell_A X)=0$ then Theorem~\ref{thmnew}(i) implies that $F_A\cong \Cell_A$ and hence the $t$-structure associated with $P_A$ is defined by the $\Sigma A$-null objects and the $A$-cellular objects.

\begin{lem}
If $P_A$ commutes with suspension, then the heart of its associated $t$-structure is trivial.
\end{lem}
\begin{proof}
If $P_A$ commutes with suspension, so does $\Cell_A$ and they fit into a triangle (see Theorem \ref{thmnew}).
Hence, $A$-cellular and $A$-null objects are both closed under suspensions and desuspensions.
So, if $X\ne 0$ is an object in $\T^{\ge 0}$, then $P_A X\cong X$. Thus $\Cell_A X\cong0$ and therefore $X$ cannot be in $\T^{\le 0}$.
\end{proof}
\begin{exmp}
Let $\T$ be a connective monogenic stable homotopy category as defined in~\cite[\S7]{HPS97} and denote by $S$ the unit of the monoidal structure. This is the case of the homotopy category of spectra or the derived category of a commutative ring. Then, by \cite[Proposition 7.1.2]{HPS97} the functors $\Cell_{\Sigma^k S}$ and $P_{\Sigma^k S}$ exist for all $k\in\Z$ and they are called the $k$-th connective cover functor and the $k$-th Postnikov section functor, respectively. For any object $X$ in $\T$ we have that:
$$
\T(\Sigma^n S, \Cell_{\Sigma^k S} X)=\left\{
\begin{array}{cr}
0 & \mbox{if $n<k$}, \\
\T(\Sigma^n S, X) & \mbox{if $n\ge k$},
\end{array}
\right.
$$
$$
\T(\Sigma^n S, P_{\Sigma^{k}S} X)\cong\left\{
\begin{array}{cr}
0 & \mbox{if $n\ge k$}, \\
\T(\Sigma^n S, X) & \mbox{if $n<k$}.
\end{array}
\right.
$$
Since $\T(\Sigma^{k-1}S, \Cell_{\Sigma^k S} X)=0$ for every $X$ in $\T$, we can apply Theorem \ref{thmnew} and we have a triangle
$$
\Cell_{\Sigma^k S}X\longrightarrow X\longrightarrow P_{\Sigma^k S}X\longrightarrow \Sigma\Cell_{\Sigma^k S} X.
$$
The heart of the $t$-structure associated to the nullification functor $P_{\Sigma^k S}$ is the full subcategory of $\T$ with
objects $X$ such that $\T(\Sigma^n S,X)=0$ if $n\ne k$ and it is equivalent to the category of $R$-modules, where $R$
denotes the ring $\T(S,S)$.
The objects in the heart are called \emph{Eilenberg-Mac\,Lane objects}.

Note that $\Cell_{\Sigma^k S}$ is not a triangulated functor. If $\T(\Sigma^{k-1}S, X)\ne 0$, then
$\Cell_{\Sigma^k S} \Sigma X\not\cong \Sigma \Cell_{\Sigma^k S}X$.
\label{expos}
\end{exmp}

\section{Cellularization of ring structures and module structures}
\label{cellringsmods}
Let $\T$ be a stable homotopy category in the sense of~\cite{HPS97}. That is, $\T$ is a triangulated category with a closed symmetric monoidal structure compatible with the triangulation (see~\cite[A.2]{HPS97}). We denote by $S$ be the unit of the monoidal structure, by $\otimes$ the tensor product and by $F(-,-)$ the internal hom in $\T$, right adjoint to the tensor product. For every $X$, $Y$ and $Z$ in $\T$ there is a natural isomorphism
$$
\T(X\otimes Y, Z)\cong \T(X, F(Y, Z)).
$$

An object $X$ in $\T$ is \emph{small} if the natural map
$\oplus_i\T(X, Y_i)\rightarrow \T(X, \coprod_i Y_i) $ is an
isomorphism of abelian groups. A set of objects $\mathcal{G}$ \emph{generates} $\T$
if the smallest triangulated subcategory of $\T$ that contains
$\mathcal{G}$ and is closed under coproducts and retracts is $\T$
itself. In particular, an object $X$ in $\T$ is zero if and only if $\T(\Sigma^k G, X)=0$ for all $k\in\Z$ and all $G\in\mathcal{G}$.

In this section, we assume that $\T$ is a connective monogenic
stable homotopy category~\cite[\S7]{HPS97}, i.e., the unit $S$ is a small generator,
and $\T(\Sigma^k S, S)=0$ for all $k<~0$. Our motivating examples
are the homotopy category of spectra ${\rm Ho}^s$ and the
(unbounded) derived category $\mathcal{D}(R)$ of a commutative
ring with unit $R$. But there are other examples of connective
monogenic stable homotopy categories, such as the category
$\mathcal{C}(B)$ of (unbounded) chain complexes of injective
$B$-comodules, where $B$ is a commutative Hopf algebra over a
field, and the morphisms are given by cochain homotopy classes of
maps; see~\cite[1.2.3(d)]{HPS97}.

A \emph{ring} $(E,\mu,\eta)$ in $\T$ is an object $E$ equipped with two morphism
$\mu\colon E\otimes E\rightarrow E$ and $\eta\colon S\rightarrow E$ such that the following diagrams commute:
$$
\xymatrix{
E\otimes E\otimes E\ar[r]^-{1\otimes\mu}\ar[d]_{\mu\otimes 1} & E\otimes E\ar[d]^{\mu} \\
E\otimes E \ar[r]_-{\mu} & E
}
\qquad
\xymatrix{
S\otimes E\ar[r]^{\eta\otimes 1}\ar[dr] & E\otimes E\ar[d]^{\mu} & E\otimes S\ar[l]_{1\otimes\eta}\ar[dl]\\
 & E. & }
$$
Given a ring $E$, a (left) \emph{$E$-module} $(M,m)$ in $\T$ is an object $M$ together with a map $m\colon E\otimes M\rightarrow M$ such that
the following diagrams commute:
$$
\xymatrix{
E\otimes E\otimes M \ar[r]^-{\mu\otimes 1}\ar[d]_{1\otimes m} & E\otimes M\ar[d]^m \\
E\otimes M \ar[r]_-{m} & M
}
\qquad
\xymatrix{
S\otimes M\ar[r]^{\eta\otimes 1}\ar[dr] & E\otimes M\ar[d]^m & M\otimes S\ar[l]_{1\otimes \eta}\ar[dl] \\
& M. &
}
$$
A \emph{ring map} between two rings $(E,\mu,\eta)$ and
$(E',\mu',\eta')$ is a map $f\colon E\rightarrow E'$ compatible
with the structure maps $\mu$, $\eta$, $\mu'$ and $\eta'$, i.e.,
$f\circ\mu=\mu'\circ{f\otimes f}$ and $f\circ \eta=\eta'$.

A \emph{module map} between two $E$-modules $(M,m)$ and $(M',m')$
is a map $f\colon M\rightarrow M'$ compatible with the structure
maps $m$ and $m'$, i.e., $f\circ m=m'\circ(1\otimes f)$.

As proved in \cite[Propostion 7.1.2]{HPS97} the cellularization
functors $\Cell_{\Sigma^k S}$ exist for all $k\in\Z$ in any
connective monogenic stable homotopy category. We will denote the
$S$-cellularization $\Cell_{S} X$ by $X^c$ and $\Cell_S (F(-,-))$
by $F^c(-,-)$.

For every $X$ in $\T$ there is a natural map $c_X\colon
X^c\rightarrow X$ such that $\T(\Sigma^k S, c_X)$ is an
isomorphism for every $k\ge 0$ and $\T(\Sigma^k S, X)=0$ if $k<0$.
An object $X$ is called \emph{connective} if $X\cong X^c$. Using
this notation, an object $X$ is $A$-null if and only if
$F^c(A,X)=0$ and a map $g\colon X\longrightarrow Y$ is an
$A$-cellular equivalence if and only if it induces an isomorphism
$F^c(A,X)\cong F^c(A,Y)$. Note also that if $X$ is connective,
then for all $Y$ and $Z$ there is an isomorphism
$$
\T(X, F^c(Y,Z))\cong \T(X, F(Y,Z))\cong \T(X\otimes Y, Z).
$$

In \cite[Theorem 4.2 and Theorem 4.5]{CG05} we give conditions
under which a nullification functor in the homotopy category of
spectra preserves ring and module objects, namely if it commutes
with suspension or under some connectivity assumptions.
Cellularization functors \emph{do not} preserve ring structures in
general even if they commute with suspension (see~\ref{acycex}).
But they do preserve modules over connective rings.
\begin{thm}
If $E$ is a connective ring and $M$ is an $E$-module, then for any object $A$, the cellularization $\Cell_A M$ has a unique $E$-module structure such that the cellularization map $\Cell_A M\rightarrow M$ is a map
of $E$-modules.
\label{cellmod}
\end{thm}
\begin{proof}
We need a map $\overline m\colon E\otimes \Cell_A M\rightarrow \Cell_A M$ rendering the following diagram commutative in $\T$:
$$
\xymatrix{
E\otimes \Cell_A M\ar@{.>}[r]^-{\overline{m}}\ar[d]_{1\otimes c_M} & \Cell_A M\ar[d]^{c_M}\\
E\otimes M\ar[r]_-{m} & M,
}
$$
where $m$ is the structure map of the $E$-module $M$. Since $E$ is connective, we have the following chain of isomorphisms
of abelian groups
\begin{multline*}
\T(E\otimes \Cell_A M, M)\cong \T(E, F^c(\Cell_A M, M)) \\
\cong \T(E, F^c(\Cell_A M, \Cell_A M))\cong \T(E\otimes \Cell_A M, \Cell_A M).
\end{multline*}
Thus, we can extend the map $m$ to a map $\overline{m}$ which provides an $E$-module structure on~$M$.
\end{proof}
\begin{rem}
The result is also true without the connectivity assumption on the ring if the $A$\nobreakdash-cellularization functor commutes with suspension, since in this case the $A$\nobreakdash-cellular objects and the $A$\nobreakdash-cellular equivalences are both closed under suspensions and desuspensions.
\end{rem}

The study of the preservation of ring objects is more involved. Although it is not difficult to produce a
product map in the cellularization of a ring under certain connectivity conditions, the construction of a unit
needs stronger hypotheses.

Let $(E,\mu,\eta)$ be a ring object, $A$ any object in $\T$ and consider
the cellularization functor $\Cell_A$. We need to construct a
multiplication map $\overline{\mu}\colon \Cell_A E\otimes \Cell_A
E\rightarrow \Cell_A E$ such that the following diagram commutes
$$
\xymatrix{
\Cell_A E\otimes \Cell_A E \ar[d]_{c_E\otimes c_E} \ar@{.>}[r]^-{\overline{\mu}} & \Cell_A E \ar[d]^{c_E} \\
E\otimes E \ar[r]_-{\mu} & E.
}
$$
If $\Cell_A$ commutes with suspension, then we can define $\overline{\mu}$ using the following sequence of isomorphisms
\begin{multline}\notag
\T(\Cell_A E\otimes \Cell_A E, E)\cong \T(\Cell_A E, F(\Cell_A E, E)) \\ \notag
\cong \T(\Cell_A E, F(\Cell_A E, \Cell_A E))\cong \T(\Cell_A E\otimes \Cell_A E, \Cell_A E),
\end{multline}
since the map $\Cell_A E\rightarrow E$ is an $A$-cellular
equivalence and $\Cell_A E$ is $A$-cellular. A~similar argument
holds when $\Cell_A E$ is connective (but $\Cell_A$ does not
necessarily commute with suspension), since in this case
$$
\T(\Cell_A E\otimes \Cell_A E, E)\cong \T(\Cell_A E, F^c(\Cell_A E, E)).
$$
So, there is no problem in constructing $\overline{\mu}$ when $\Cell_A$ commutes with suspension or when $\Cell_A E$ is connective. However, for the unit of $\Cell_A E$, we need a canonical map $\overline{\eta}\colon S\rightarrow \Cell_A E$ such that the following
diagram commutes:
$$
\xymatrix{
S \ar@{.>}[r]^-{\overline{\eta}}\ar[dr]_{\eta} & \Cell_A E \ar[d]^{c_E} \\
 & E.
}
$$
Let $C_A E$ be a cofibre of the cellularization map $\Cell_A
E\rightarrow E$. We have a triangle
$$
\Cell_A E\longrightarrow E\longrightarrow C_A E\longrightarrow \Sigma \Cell_A E
$$
yielding an exact sequence of abelian groups,
\begin{multline}\notag
\cdots\longrightarrow \T(\Sigma S, E)\longrightarrow \T(\Sigma S,C_A E)
\longrightarrow \T(S, \Cell_A E) \\
\longrightarrow \T(S, E)\longrightarrow \T(S,C_A E)
\longrightarrow\T(\Sigma^{-1}S, \Cell_A E)\longrightarrow\cdots
\end{multline}
Hence, there is an isomorphism $\T(S, \Cell_A E)\cong \T(S, E)$ if and only if $\T(\Sigma S, E)\rightarrow \T(\Sigma S, C_A E)$ is surjective and
$\T(S, C_A E)\rightarrow\T(\Sigma^{-1} S, \Cell_E A)$ is injective. In what follows we will use the notation $\pi_k X=\T(\Sigma^k S, X)$ for all $k\in\Z$.

\begin{thm}
Let $E$ be a ring object and $A$ any object in $\T$. Assume that either one of the following holds:
\begin{itemize}
\item[{\rm (i)}] $\Cell_A$ commutes with suspension, the morphism $\pi_1(E)\twoheadrightarrow \pi_1(P_A E)$ is surjective
and the morphism $\pi_0(P_A E)\rightarrowtail \pi_{-1}(\Cell_E A)$ is injective, or
\item[{\rm (ii)}] $\Cell_A E$ is connective, $\Cell_A$ is the colocalization functor associated to a nullification functor $P_B$ for some $B$ (see Definition~\ref{assoc_loc_coloc}), the morphism $\pi_1(E)\twoheadrightarrow \pi_1(P_B E)$
is surjective and $\pi_0(P_B E)=0$.
\end{itemize}
Then, $\Cell_A E$ has a unique ring structure such that the cellularization map is a map of rings.
\label{cellrings}
\end{thm}
\begin{proof}
The first part is proved by using the previous discussion and the triangle $\Cell_A E\rightarrow E\rightarrow P_A E\rightarrow\Sigma\Cell_A E$. For the second, use the triangle $\Cell_A E\rightarrow E\rightarrow P_B E\rightarrow \Sigma\Cell_A E$.
\end{proof}

\begin{exmp}
Let $\T$ be a connective monogenic stable homotopy category, $A=S$ and $E$ any ring object in $\T$. Thus, $\Cell_A E$ is the connective cover of $E$ (see Example ~\ref{expos}) and there is a triangle
$$
\Cell_S E\longrightarrow E\longrightarrow P_S E\longrightarrow \Sigma \Cell_S E,
$$
where $P_S$ is the $0$-th Postnikov section functor. Therefore $\pi_1 P_S E\cong\pi_0 P_S E=0$ and by part (ii) of Theorem~\ref{cellrings} we have that if $E$ is a ring object, then so is its connective cover $\Cell_{S}E$. Moreover by Theorem~\ref{cellmod}, if $M$ is an $E$-module, then the $k$-th connective cover $\Cell_{\Sigma^k S} M$ is a $\Cell_S E$-module for all $k\in \Z$, since any $E$-module is canonically a $\Cell_S E$\nobreakdash-module.
\end{exmp}

\section{Cellularization of Eilenberg-Mac\,Lane spectra}
Throughout this section, $\T$ will denote the homotopy category of spectra ${\rm Ho}^s$ (see for example~\cite{Ada74}). This is a connective monogenic stable homotopy category. The tensor product in ${\rm Ho}^s$ is the smash product, denoted by $\wedge$, and the unit of the monoidal structure is the sphere spectrum $S$.

If $G$ is any abelian group and $k\in\Z$, we denote by $\Sigma^k HG$ the associated \emph{Eilenberg-Mac\,Lane spectrum}, i.e., $\pi_n\Sigma^k HG=G$ if $k=n$ and zero otherwise.
In this section, we show that cellularizations in ${\rm Ho}^s$ preserve stable $R$\nobreakdash-GEMs and study the particular case of the cellularization of a $k$-th suspension
of an Eilenberg-Mac\,Lane spectrum. We  prove that this cellularization is either zero or has at most two non trivial homotopy groups in
consecutive dimensions $k$ and~$k-1$.

\begin{defn}
An object $X$ in $\T$ is a \emph{stable GEM} if $X\cong\vee_{k\in\Z}\Sigma^ k HB_k$.
If the abelian groups $B_k$ are $R$-modules for some ring $R$ with unit,
we say that $X$ is a \emph{stable $R$-GEM}.
\end{defn}

In \cite[Proposition 5.3]{CG05} we proved that stable $R$-GEMs are the same thing as $HR$-modules, hence we can apply Theorem~\ref{cellmod}, since $HR$ is a connective ring spectrum, and obtain the following.

\begin{cor}
If $X$ is a stable $R$-GEM and $A$ is any object in $\T$, then $\Cell_A X$ is a stable $R$-GEM and the cellularization map is a map of
$HR$-modules. $\hfill\qed$
\end{cor}

Now we are going to consider the particular case $X=\Sigma^n HG$ for some $n\in\Z$, where $G$ is any $R$-module. For every $i\in\Z$ we have the following sequence of maps
$$
\Sigma^iHB_i\longrightarrow \Cell_A \Sigma^n HG\cong\vee_{k\in\Z}\Sigma^k HB_k\stackrel{c}{\longrightarrow} \Sigma^n HG,
$$
where the first map is the inclusion of each factor into the coproduct and the second is the cellularization map. Note that all
the maps are maps of $H\Z$-modules, since any $HR$-module is canonically an $H\Z$-module. By \cite[Corollary 5,4]{CG05}, we have that
$$
\T(\Sigma^i HB_i, \Sigma^n HG)_{{H\Z{\mbox{-}\mathrm{mod}}}}=0
$$
unless $i=n$ or $i=n-1$, where $\T(-,-)_{{H\Z{\mbox{-}\mathrm{mod}}}}$ denotes the set of $H\Z$\nobreakdash-module maps; see \cite{Gut05} for more details. Now, using the universal property of $\Cell_A$ and the fact that $\Sigma^i HB_i$ is a retract of $\vee_{k\in\Z}\Sigma^k HB_k$ and hence $A$-cellular, we get
that $B_i=0$ if $i\ne n$ or $i\ne n-1$. We obtain the following result (cf. \cite[Theorem~5.6]{CG05})
\begin{thm}
Let $G$ be any abelian group, $n\in\Z$ and $A$ be any object in $\T$. Then $\Cell_A \Sigma^n HG\cong \Sigma^{n-1}HB\vee \Sigma^n HC$ for some
abelian groups $B$ and $C$. If $G$ is an $R$-module, then $B$ and $C$ are also $R$-modules. $\hfill\qed$
\end{thm}

We can obtain information about the groups $B$ and $C$ by using
the universal property of the cellularization functor. Since
$\Cell_A \Sigma^n HG\cong \Sigma^{n-1}HB\vee \Sigma^n HC$, we have
that $\Sigma^{n-1}HB$ and $\Sigma^n HC$ are both $A$-cellular,
because they are a retract of an $A$-cellular object. We have the
following commutative diagrams
$$
\xymatrix{
\Sigma^{n-1}HB\vee \Sigma^n HC \ar[r]^-c & \Sigma^n HG \\
\Sigma^{n-1} HB \ar[ur] \ar[u] &
}
\qquad
\xymatrix{
\Sigma^{n-1}HB\vee \Sigma^n HC \ar[r]^-c & \Sigma^n HG \\
\Sigma^{n} HC \ar[ur] \ar[u] &
}
$$
that together with the universal property of the cellularization map give rise to the isomorphisms of abelian groups
\begin{gather}\notag
\T(\Sigma^{n-1}HB, \Sigma^{n-1}HB)\times \T(\Sigma^{n-1}HB, \Sigma^n HC)\cong \T(\Sigma^{n-1} HB, \Sigma^n HG),\\ \notag
\T(\Sigma^n HC, \Sigma^{n-1} HB)\times \T(\Sigma^n HC, \Sigma^n HC)\cong \T(\Sigma^n HC, \Sigma^n HG).
\end{gather}
By the same argument, using that $\Sigma^n HB$ is also $A$-cellular, since it is the suspension of an $A$-cellular object, we obtain the isomorphism
$$
\T(\Sigma^{n}HB, \Sigma^{n-1}HB)\times \T(\Sigma^{n}HB, \Sigma^n HC)\cong \T(\Sigma^{n} HB, \Sigma^n HG).
$$
\begin{prop}
Let $A$ be any object in $\T$ and $n\in\Z$. If $B\cong\pi_{n-1}\Cell_A\Sigma^n HG$ and $C\cong\pi_n \Cell_A\Sigma^n HG$, then
the following hold:
\begin{itemize}
\item[{\rm (i)}] $\Hom(B,B)\oplus \Ext(B,C)\cong \Ext(B,G)$,
\item[{\rm (ii)}] $\Hom(C,C)\cong \Hom(C,G)$,
\item[{\rm (iii)}] $\Hom(B,C)\cong \Hom(B,G)$.
\end{itemize}
\label{propo01}
\end{prop}
\begin{proof}
The result follows by using the previous isomorphisms and the fact that $\T(HB, HC)=\Hom(B,C)$, $\T(HB, \Sigma HC)=\Ext(B,C)$ and $\T(\Sigma HB, HC)=0$
for any abelian groups $B$ and~$C$.
\end{proof}

\begin{cor}
If $G$ is a divisible abelian group, then $\Cell_A\Sigma^n HG$ is either zero or $\Sigma^n HC$ for some abelian group $C$. Moreover,
this group $C$ satisfies the relation $\Hom(C,C)\cong \Hom(C,G)$.
\end{cor}
\begin{proof}
This follows from Proposition \ref{propo01}, since if $G$ is divisible, $\Ext(B,G)=0$ for any abelian group~$B$.
\end{proof}

\begin{exmp}
Let $p$ be a prime and consider the Eilenberg-Mac\,Lane spectrum $H\Z/p$, where $\Z/p$ denotes the integers modulo $p$. We know that the homotopy groups of $\Cell_A H\Z/p$ are
$\Z/p$-modules and thus, they split as a direct sum of copies of $\Z/p$. Since suspensions and retracts of $A$-cellular objects are again
$A$\nobreakdash-cellular, we have that either $\Cell_A H\Z/p\cong H\Z/p$ or zero. If $\Cell_A H\Z/p$ is zero, the triangle
$$
H\Z/p^{r-1}\longrightarrow H\Z/p^r\longrightarrow H\Z/p\rightarrow \Sigma H\Z/p^{r-1}
$$
for every $r\ge 2$, implies by induction that $\Cell_A H\Z/p^r=0$ for every $r\ge 2$ by Proposition \ref{propos02}.

Assume now that $H\Z/p$ is $A$-cellular. Given a positive integer
$k$, we know that $\Cell_A H\Z/p^k\cong \Sigma^{-1}HB\vee HC$,
where $B$ and $C$ are now $\Z/p^k$-modules and thus they split as
a direct sum of subgroups isomorphic to $\Z/p^j$, with $1\le j\le
k$ (see for example \cite{Kap69}). Suppose that $B\ne 0$, and let
$m$ be the smallest natural number such that $\Z/p^m$ is a direct
summand of $B$. Then $\Sigma^{-1}H\Z/p^m$ and $H\Z/p^m$ are
$A$-cellular. The triangle
$$
\Sigma^{-1}H\Z/p^m\longrightarrow H\Z/p^{rm}\longrightarrow H\Z/p^{(r+1)m}\rightarrow H\Z/p^m
$$
for every $r\ge 1$ imply inductively that $H\Z/p^{rm}$ is $A$-cellular for every $r\ge 1$ by Lemma \ref{suseq}. Take $r\ge 1$ such that
$rm>k$. Then, since $H\Z/p^{rm}$ is $A$-cellular, we infer that $H\Z/p^k$ is $A$-cellular by using the triangle
$$
H\Z/p\longrightarrow H\Z/p^{i}\longrightarrow H\Z/p^{i-1}\rightarrow\Sigma H\Z/p.
$$
This is a contradiction, therefore $B=0$.

To compute $C$, we use Proposition \ref{propo01}(ii). We have an isomorphism
$$
\Hom(C,C)\cong \Hom(C,\Z/p^k).
$$
Since $C$ is a direct sum of copies of $\Z/p^j$, with $1\le j\le k$, this forces $C=\Z/p^j$, with $1\le j\le k$.
Hence for any object $A$ in $\T$, we have that $\Cell_A H\Z/p^k$ is either zero or $H\Z/p^j$ with $1\le j\le k$.
\end{exmp}
\begin{exmp}
Let $A=H\Z/p^k$, and $n\ge k$. The map $H\Z/p^k\longrightarrow H\Z/p^n$ induced by the map $\Z/p^k\longrightarrow \Z/p^n$ sending $1\mapsto p^{n-k}$
is an $A$-cellular equivalence, so $\Cell_{H\Z/p^k}H\Z/p^n=H\Z/p^k$. This example in the case $k=1$ and $n=2$ shows that $\Cell_A$ is not semiexact in
general. If it were semiexact, then the triangle
$$
H\Z/p\longrightarrow H\Z/p^2\longrightarrow H\Z/p\longrightarrow \Sigma H\Z/p
$$
would imply that $H\Z/p^2$ is $H\Z/p$-cellular, and that is a contradiction.

If $n\le k$, then the map $\Z/p^k\longrightarrow \Z/p^k$ sending $1\mapsto p^n$ yields a triangle
$$
H\Z/p^k\longrightarrow H\Z/p^k\longrightarrow H\Z/p^n\vee\Sigma H\Z/p^n
\longrightarrow \Sigma H\Z/p^k.
$$
So $H\Z/p^n$ is $H\Z/p^k$-cellular since it is a retract of a
cofibre of a map between $H\Z/p^k$-cellular objects. Hence,
$\Cell_{H\Z/p^k}H\Z/p^n=H\Z/p^n$. \label{notqe}
\end{exmp}
The above examples can be generalized when instead of $H\Z/p^n$ we take an arbitrary suspension $\Sigma^m H\Z/p^n$ for any $m\in\Z$.
\begin{prop}
For any $m\in\Z$ and any $n\ge 0$, $\Cell_A \Sigma^m
H\Z/p^n\cong\Sigma^m H\Z/p^j$, where $1\le j\le n$. If $A=\Sigma^m
H\Z/p^k$, then $\Cell_A \Sigma^m H\Z/p^n\cong \Sigma^m H\Z/p^k$ if
$n\ge k$ or $\Cell_A \Sigma^m H\Z/p^n\cong \Sigma^m H\Z/p^n$ if
$n<k$. $\hfill\qed$ \label{propop}
\end{prop}

\subsection{Bousfield's acyclizations}
Let $E$ be any spectrum and consider now the
Bousfield homological localization functor $L_E$. In \cite{Bou79a},
Bousfield proved that for every spectrum $E$, there is another
spectrum $A$, such that homological localization with respect to $E$
is the same as nullification with respect to $A$, i.e., $L_E X\cong
P_A X$ for every $X$. Since $L_E$ is a triangulated functor, for
every spectrum $X$ there is a triangle
$$
\Cell_A X\longrightarrow X\longrightarrow P_A X\longrightarrow \Sigma\Cell_A X,
$$
where, using the terminology of \cite{Bou79a}, $\Cell_A$ is called the \emph{$E_*$-acyclization} functor.

In \cite{Gut10}, we study the effect of homological
localization functors on Eilenberg\nobreakdash-Mac\,Lane spectra and give a method to compute the localization $L_E HG$ for any spectrum $E$ and any abelian group $G$. The results depend basically on the $E$\nobreakdash-acy\-cli\-ci\-ty patterns of the spectra $H\Q$ and $H\Z/p$ for each prime $p$.

In the following computations, we will always assume that the
nullification functor $P_A$ is equivalent to some homological
localization $L_E$, and hence $\Cell_A$ is the corresponding
$E_*$\nobreakdash-acyclization functor.

\subsubsection{Acyclizations of $H\Z$}
\label{acycex} In \cite[\S5]{Gut10}, we prove
that the only possible homological localizations of the
Eilenberg-Mac\,Lane spectrum $H\Z$ are $0$, $H\Z_P$ or
$\prod_{p\in P}H\widehat{\Z}_p$, where $P$ is a set of primes,
$\Z_P$ denotes the integers localized at the set  $P$, and
$\widehat{\Z}_p$ denotes the $p$-adic integers (see
\cite[Proposition 5.2]{Gut10}). In the case of acyclizations, we
have the following possibilities. If $P_A H\Z=0$ then $\Cell_A
H\Z\cong H\Z$ and if $P_A H\Z\cong H\Z$, then $\Cell_A H\Z=0$. If
$P_A H\Z\cong H\Z_P$, then there is a triangle
$$
\Sigma^{-1} H(\oplus_{p\in P'}\Z/p^{\infty})\longrightarrow H\Z\longrightarrow H\Z_P\longrightarrow H(\oplus_{p\in P'}\Z/p^{\infty}),
$$
where $P'$ is the complement of the set of primes $P$. So, in this case $\Cell_A H\Z=\Sigma^{-1}H(\oplus_{p\in P'}\Z/p^{\infty})$.
Finally, if $P_A H\Z=\prod_{p\in P}H\widehat{\Z}_p$, then there is a triangle
$$
\Sigma^{-1} H\left(\left(\prod\nolimits_{p\in P}\widehat{\Z}_p\right)/\Z\right)
\longrightarrow H\Z\longrightarrow \prod\nolimits_{p\in P}H\widehat{\Z}_p\longrightarrow H\left(\left(\prod\nolimits_{p\in P}\widehat{\Z}_p\right)/\Z\right).
$$
Hence, $\Cell_A H\Z\cong \Sigma^{-1} H\left(\left(\prod\nolimits_{p\in P}\widehat{\Z}_p\right)/\Z\right)$. These computations provide an example of a cellularization functor that commutes with suspension but does not preserve (connective) ring spectra.

\subsubsection{Acyclizations of $H\Z/p^k$}
In\cite[\S5]{Gut10}, we prove that $P_A H\Z/p^k=0$ or $H\Z/p^k$, hence $\Cell_A H\Z/p^k$ is either $\Z/p^k$ or $0$, respectively
(cf. Proposition \ref{propop}).

\subsubsection{Acyclizations of $H\Z/p^{\infty}$}
In \cite[\S5]{Gut10}, we prove that $P_A H\Z/p^{\infty}=0$, or $H\Z/p^{\infty}$,  or $\Sigma H\widehat{\Z}_p$. In the
first and second case, we have that $\Cell_A H\Z/p^{\infty}$ is $H\Z/p^{\infty}$ or $0$, respectively. In the third case, we can use
the triangle
$$
H\widehat{\Q}_p\longrightarrow H\Z/p^{\infty}\longrightarrow \Sigma H\widehat{\Z}_p\longrightarrow \Sigma H\widehat{\Q}_p,
$$
to conclude that $\Cell_A H\Z/p^{\infty}\cong H\widehat{\Q}_p$, where $H\widehat{\Q}_p$ denotes the $p$-adic rationals.

\end{document}